\documentclass[twoside]{article}
\usepackage{amsmath,amsthm,amssymb,bm}
\usepackage[reset,a4paper,vmargin=3truecm,hmargin=3truecm]{geometry}

\numberwithin{equation}{section}

\theoremstyle{plain}
\newtheorem{thm}{Theorem}[section]
\newtheorem{lem}[thm]{Lemma}
\newtheorem{cor}[thm]{Corollary}
\theoremstyle{remark}
\newtheorem{rem}[thm]{Remark}
\theoremstyle{definition}
\newtheorem{ex}[thm]{Example}

\newcommand{\adet}[1][\alpha]{\operatorname{det}_{#1}}
\DeclareMathOperator{\Tr}{Tr}
\DeclareMathOperator{\wrdet}{wrdet}

\DeclareMathOperator{\per}{per}
\DeclareMathOperator{\sgn}{sgn}

\DeclareMathOperator{\Imm}{Imm}
\DeclareMathOperator{\STab}{STab}
\newcommand{\sym}[1]{\mathfrak{S}_{#1}}
\newcommand{\1}{\bm{1}}

\newcommand{\C}{\mathbb{C}}

\newcommand{\va}{\bm{a}}
\newcommand{\ve}{\bm{e}}
\newcommand{\Smod}[1]{\bm{S}^{#1}}
\newcommand{\I}{\mathbb{I}}
\newcommand{\gind}[2]{\left|#1:#2\right|}

\newcommand{\len}[1]{l(#1)}
\newcommand{\mat}[1]{\begin{pmatrix}#1\end{pmatrix}}
\newcommand{\tr}[1]{\mathord{\mathopen{{\vphantom{#1}}^t}\!#1}}
\newcommand{\set}[2]{\left\{#1\,\middle\vert\,#2\right\}}

\newcommand{\Grand}[2]{\thanks{Partially supported by Grand-in-Aid for Scientific Research (#1) No. #2.}}

\begin{document}

\title{\sffamily Averages of alpha-determinants over permutations}
\author{Kazufumi KIMOTO\Grand{C}{25400044}}
\date{March 15, 2014}
\pagestyle{myheadings}
\markboth{K.~KIMOTO}{Averages of alpha-determinants over permutations}

\maketitle

\begin{abstract}
We show that certain weighted average of the $\alpha$-determinant of a $kn$ by $kn$ matrix of the form $A\otimes\1_{1,k}$,
the Kronecker product of a $kn$ by $n$ matrix $A$ and $1$ by $k$ all one matrix $\1_{1,k}$,
over permutations of $kn$ letters is reduced to the $k$-wreath determinant of $A$ up to constant.
The constant is exactly given by the modified content polynomial
for the Young diagram $(k^n)$.
As a corollary, we give a `determinantal' formula for certain functions on the symmetric groups
which are invariant under the left and right translation by a Young subgroup,
especially the values of the Kostka numbers for rectangular shapes with arbitrary weight.
This corollary gives a generalization of the formula of irreducible characters of the symmetric group
for rectangular shapes due to Stanley.
\end{abstract}

\section{Introduction}

The \emph{$\alpha$-determinant} of an $N$ by $N$ square matrix $A=(a_{ij})$
is defined as a parametric deformation of the usual determinant as
\begin{equation*}
\adet A:=\sum_{\sigma\in\sym N}\alpha^{\nu(\sigma)}a_{\sigma(1)1}a_{\sigma(2)2}\dots a_{\sigma(N)N},
\end{equation*}
where $\alpha$ is a complex parameter and $\nu(\sigma)$ for a permutation $\sigma\in\sym N$ is defined to be
$N$ minus the number of disjoint cycles in $\sigma$.
By definition, we see that
\begin{equation*}
\adet[-1]A=\det A,\qquad
\adet[1]A=\per A,\qquad
\adet[0]A=a_{11}a_{22}\dots a_{NN},
\end{equation*}
where $\per A$ is the \emph{permanent} of $A$.
It is Vere-Jones \cite{VJ1988} who first introduce such a parametric deformations,
which he called the \emph{$\alpha$-permanent}.
Here we adopt the modified definition and terminology by Shirai and Takahashi \cite{ST2003}.
The $\alpha$-determinant is multiplicative only if $\alpha=-1$.

Let $P(\sigma)=(\delta_{i\sigma(j)})$ be the permutation matrix for a permutation $\sigma\in\sym N$.
The sum
\begin{equation}\label{eq:unsigned average of adet}
\sum_{\sigma\in\sym k}\adet\bigl(AP(\sigma)\bigr)
\end{equation}
is a polynomial in $\alpha$ which is divisible by $(1+\alpha)\dots(1+(k-1)\alpha)$
for a given $N$ by $N$ matrix $A$.
Here we regard $\sym k$ as a subgroup of $\sym N$
consisting of permutations which do not move the $N-k$ letters $k+1,k+2,\dots,N$.
This fact is used to show that
the $\alpha$-determinant is weakly alternating
when $\alpha$ is a reciprocal of a negative integer
in the sense that
$\adet[-1/k]A$ vanishes whenever more than $k$ columns or rows in $A$ are equal (Lemma \ref{lem:weak alternating}).
Based on this fact,
we define the \emph{$k$-wreath determinant} $\wrdet_kA$ of a $kn\times n$ matrix $A$ by
\begin{equation*}
\wrdet_kA:=\adet[-1/k]
\bigl(\overbrace{\va_1,\dots,\va_1}^k,\dots,\overbrace{\va_n,\dots,\va_n}^k\bigr),
\end{equation*}
where $\va_j$ is the $j$-th column vector of $A$.
This \emph{recovers} the relative invariance
\begin{equation*}
\wrdet_k(AQ)=\wrdet_kA\,(\det Q)^k
\end{equation*}
with respect to the right translation by any $n$ by $n$ matrix $Q$ \cite{KW2008}.

In the extremal case where $k=N$, we can determine the sum \eqref{eq:unsigned average of adet} explicitly as
$$
\sum_{\sigma\in\sym N} \adet\bigl(AP(\sigma)\bigr)
=\prod_{i=1}^{N-1}(1+i\alpha)\cdot\per A.
$$
More generally, one can prove
\begin{equation}\label{eq:character weighted average of adet}
\sum_{\sigma\in\sym N} \chi^\lambda(\sigma)\adet\bigl(AP(\sigma)\bigr)
=f^\lambda f_\lambda(\alpha)\Imm_\lambda A
\end{equation}
for any partition $\lambda\vdash N$.
Especially we have
\begin{equation}\label{eq:signed average of adet}
\sum_{\sigma\in\sym N} \sgn\sigma\;\adet\bigl(AP(\sigma)\bigr)
=\prod_{i=1}^{N-1}(1-i\alpha)\cdot\det A.
\end{equation}
Here $\chi^\lambda$ is the irreducible character of $\sym N$ associated to $\lambda$,
$f^\lambda$ is the number of standard tableaux with shape $\lambda$,
$f_\lambda(\alpha)$ is the modified content polynomial for $\lambda$
and $\Imm_\lambda A$ is the immanant of $A$ associated to $\lambda$.
The identity \eqref{eq:character weighted average of adet}
is essentially equivalent to the result by Matsumoto and Wakayama \cite{MW2006}
on the irreducible decomposition of the $U(\mathfrak{gl}_N)$-cyclic submodule
generated by a single polynomial $\adet X$.
The structure of such cyclic module is the same for almost all values of $\alpha$,
but changes drastically when $\alpha$ is a reciprocal of a nonzero integer.

The purpose of the paper is to give an analog of \eqref{eq:signed average of adet}
for the $k$-wreath determinant (Theorem \ref{thm:main_result}).
As corollaries of the main result,
we also obtain a formula for certain $\sym\mu$-biinvariant functions on $\sym N$,
where $\sym\mu$ is the Young subgroup of $\sym N$ associated with a partition $\mu\vdash N$.
In particular, we get a formula for Kostka numbers with rectangular shape and arbitrary weight
(Corollaries \ref{cor:values_of_omega}, \ref{cor:values_of_zsf}).
These corollaries give a generalization of the formula
for irreducible characters of the symmetric groups associated to rectangular diagrams
which is due to Stanley \cite{S2003} (Corollary \ref{cor:values_of_chi}).

\section{Weighted averages of alpha-determinants over permutations}\label{sec:main_result}

Let $n,k$ be positive integers.
We define a linear map
$\varpi_k\colon M_{kn,n}\to M_{kn}$
by $\varpi_k(A):=A\otimes\1_{1,k}$,
where $M_{p,q}$ is the set of $p$ by $q$ complex matrices,
$M_p=M_{p,p}$ is the set of square matrices of size $p$,
$\1_{p,q}$ is the $p$ by $q$ all-one matrix,
and $\otimes$ denotes the Kronecker product of matrices
\begin{equation*}
A\otimes B=\begin{pmatrix}
a_{11}B & \dots & a_{1n}B \\ \vdots & \ddots & \vdots \\ a_{m1}B & \dots & a_{mn}B
\end{pmatrix}\qquad
(A=(a_{ij})\in M_{m,n}).
\end{equation*}
We note that $\varpi_k$ commutes with the left translation,
that is, $\varpi_k(PA)=P\,\varpi_k(A)$ for any $P\in M_{kn}$ and $A\in M_{kn,n}$.
We also notice that
\begin{equation}\label{eq:equivariance}
\varpi_k(A)P(g)=\varpi_k(A)
\end{equation}
for any $A\in M_{kn,n}$ and $g\in\sym k^n=\sym{(k^n)}$.
For a $kn$ by $n$ matrix $A\in M_{kn,n}$, the \emph{$k$-wreath determinant} of $A$ is defined by
\begin{equation}
\wrdet_kA:=\adet[-1/k]\varpi_k(A).
\end{equation}
The $1$-wreath determinant is the ordinary determinant: $\wrdet_1A=\det A$.
See \cite{KW2008} for basic facts on the wreath determinants.

\begin{ex}[$n=k=2$]
For $A=(a_{ij})\in M_{4,2}$, the $2$-wreath determinant of $A$ is
\begin{align*}
\wrdet_2A
&=\adet[-1/2]\mat{a_{11} & a_{11} & a_{12} & a_{12} \\
a_{21} & a_{21} & a_{22} & a_{22} \\
a_{31} & a_{31} & a_{32} & a_{32} \\
a_{41} & a_{41} & a_{42} & a_{42}} \\
&=\frac14\Bigl\{a_{11}a_{21}a_{32}a_{42}+a_{12}a_{22}a_{31}a_{41}\Bigr\} \\
&\qquad\qquad{}-\frac18\Bigl\{a_{11}a_{22}a_{31}a_{42}+a_{11}a_{22}a_{32}a_{41}
+a_{12}a_{21}a_{31}a_{42}+a_{12}a_{21}a_{32}a_{41}\Bigr\}.
\end{align*}
We can express $\wrdet_2A$ as a sum of products of minor determinants of $A$ as
\begin{equation*}
\wrdet_2A=\frac18\begin{vmatrix}
a_{11} & a_{12} \\ a_{31} & a_{32}
\end{vmatrix}\begin{vmatrix}
a_{21} & a_{22} \\ a_{41} & a_{42}
\end{vmatrix}
+\frac18\begin{vmatrix}
a_{11} & a_{12} \\ a_{41} & a_{42}
\end{vmatrix}\begin{vmatrix}
a_{21} & a_{22} \\ a_{31} & a_{32}
\end{vmatrix},
\end{equation*}
which apparently shows the relative invariance $\wrdet_2(AQ)=\wrdet_2A\,(\det Q)^2$ for $Q\in M_2$.
\end{ex}

For a partition $\lambda$,
$f_\lambda(\alpha)$ is the \emph{modified content polynomial} for $\lambda$
\begin{equation*}
f_\lambda(x)=\prod_{(i,j)\in\lambda}(1+(j-i)x),
\end{equation*}
where we identify $\lambda$ with its corresponding Young diagram.
For instance, we have
\begin{equation*}
f_{(N)}(x)=\prod_{i=1}^{N-1}(1+ix),\qquad
f_{(1^N)}(x)=\prod_{i=1}^{N-1}(1-ix)
\end{equation*}
for any positive integer $N$.
It is notable that
\begin{equation}\label{eq:adet of all-one matrix}
\adet\1_N=\sum_{\sigma\in\sym N}\alpha^{\nu(\sigma)}=f_{(N)}(\alpha),
\end{equation}
where $\1_N=\1_{N,N}$.

Our goal is to prove the
\begin{thm}\label{thm:main_result}
For each positive integer $k$, the equality
\begin{equation}\label{eq:average_formula}
\sum_{\sigma\in\sym{kn}}\Bigl(-\frac1k\Bigr)^{\nu(\sigma)}\adet\bigl(\varpi_k(A)P(\sigma)\bigr)
=f_{(k^n)}(\alpha)\wrdet_kA
\end{equation}
holds.
\end{thm}
When $k=1$, the theorem is reduced to the equality \eqref{eq:signed average of adet}.

\section{Proof of the theorem}

For later use, we put $\I_{n,k}=I_n\otimes\1_{k,1}$.
We postpone the proofs of the lemmas used in this section to \S \ref{sec:proofs of lemmas}.

\subsection{Reduction}

To prove the theorem, we need the characterization of the $k$-wreath determinant.
\begin{lem}[Corollary 5.8 in \cite{KW2008}]\label{lem:char_of_wrdet}
Suppose that a function $f\colon M_{kn,n}\to\C$ satisfies the following conditions.
\begin{itemize}
\item[{\upshape (W1)}] $f$ is multilinear in row vectors.
\item[{\upshape (W2)}] $f(AQ)=f(A)\,(\det Q)^k$ for any $Q\in M_n$.
\item[{\upshape (W3)}] $f(P(g)A)=f(A)$ for any $g\in\sym k^n$.
\end{itemize}
Then $f$ is equal to the $k$-wreath determinant $\wrdet_k$ up to constant multiple.
\end{lem}
To determine the constant factor explicitly in our discussion below, the formula
\begin{equation*}
\wrdet_k\I_{n,k}
=\adet[-1/k](I_n\otimes\1_k)
=\Bigl(\frac{k!}{k^k}\Bigr)^n
\end{equation*}
is useful (see Lemma 4.6 in \cite{KW2008}).

Let $F(\alpha;A)$ be the left-hand side of \eqref{eq:average_formula}.
Since the multilinearlity of $F(\alpha;A)$ in row vectors of $A$ is obvious by its definition,
we have only to show the following three equations to obtain the theorem.
\begin{align}
F(\alpha;AQ)&=F(\alpha;A)(\det Q)^k\qquad(Q\in M_n), \tag{A}\label{eq:to_be_proved_1} \\
F(\alpha;P(g)A)&=F(\alpha;A)\qquad(g\in\sym k^n), \tag{B}\label{eq:to_be_proved_2} \\
F(\alpha;\I_{n,k})&=\Bigl(\frac{k!}{k^k}\Bigr)^nf_{(k^n)}(\alpha). \tag{C}\label{eq:to_be_proved_3}
\end{align}

We introduce the two-parameter deformation of the determinant as
\begin{equation}
\adet[\alpha,\beta]A
:=\sum_{\tau,\sigma\in\sym N}\alpha^{\nu(\tau)}\beta^{\nu(\sigma)}\prod_{i=1}^N a_{\tau(i)\sigma(i)}.
\end{equation}
It is clear that this is symmetric in $\alpha$ and $\beta$,
i.e. $\adet[\alpha,\beta]A=\adet[\beta,\alpha]A$.
Notice that
\begin{equation*}
F(\alpha;A)
=\adet[\alpha,-1/k]\bigl(\varpi_k(A)\bigr)
=\sum_{\sigma\in\sym{kn}}\alpha^{\nu(\sigma)}\adet[-1/k]\bigl(\varpi_k(A)P(\sigma)\bigr).
\end{equation*}

\begin{rem}\label{rem:averages_of_adet}
The equalities \eqref{eq:unsigned average of adet} and \eqref{eq:signed average of adet} are readily obtained
from the symmetry $\adet[\alpha,\pm1]A=\adet[\pm1,\alpha]A$.
\end{rem}




\subsection{Proofs of \eqref{eq:to_be_proved_1} and \eqref{eq:to_be_proved_2}}

Let $\va_1,\dots,\va_n\in\C^{kn}$.
We have only to prove \eqref{eq:to_be_proved_1} when $Q$ is an elementary matrix.
Namely, it is suffice to verify
\begin{align}
F(\alpha;(\va_1,\dots,\va_j+c\va_i,\dots,\va_n))
&=F(\alpha;(\va_1,\dots,\va_j,\dots,\va_n))\qquad(i\ne j), \label{eq:kihon1} \\
F(\alpha;(\va_1,\dots,c\va_j,\dots,\va_n))
&=c^kF(\alpha;(\va_1,\dots,\va_j,\dots,\va_n))\qquad(c\in\C). \label{eq:kihon2}
\end{align}
The equation \eqref{eq:kihon2} obviously follows from
the definition of $F(\alpha;A)$ and the multilinearity of the $\alpha$-determinant in column vectors.
The equation \eqref{eq:kihon1} is guaranteed by the following lemma.

\begin{lem}[Lemma 2.3 in \cite{KW2008}]\label{lem:weak alternating}
Let $N$ and $k$ be positive integers such that $k<N$.
If more than $k$ column vectors in $A\in M_N$ are equal, then $\adet[-1/k]A=0$.
\end{lem}

The second equality \eqref{eq:to_be_proved_2} is shown
by using \eqref{eq:equivariance}, \eqref{eq:to_be_proved_1} and
the elementary fact
\begin{equation}\label{eq:commute_with_P}
\adet\bigl(P(\sigma)A\bigr)=\adet\bigl(AP(\sigma)\bigr)\qquad(A\in M_N,\ \sigma\in\sym N)
\end{equation}
as follows: For any $g\in\sym k^n$, we have
\begin{align*}
F(\alpha;P(g)A)
&=\sum_{\sigma\in\sym{kn}}\alpha^{\nu(\sigma)}\adet[-1/k]\Bigl(\varpi_k\bigl(P(g)A\bigr)P(\sigma)\Bigr)
\\
&=\sum_{\sigma\in\sym{kn}}\alpha^{\nu(\sigma)}\adet[-1/k]\Bigl(P(g)\varpi_k\bigl(A\bigr)P(\sigma)\Bigr)
\\
&=\sum_{\sigma\in\sym{kn}}\alpha^{\nu(g^{-1}\sigma g)}\adet[-1/k]\Bigl(\varpi_k\bigl(A\bigr)P(g)P(g^{-1}\sigma g)\Bigr)
=F(\alpha;A).
\end{align*}

\subsection{Proof of \eqref{eq:to_be_proved_3}}

Let $N$ be a positive integer.
For each partition $\lambda\vdash N$,
$\chi^\lambda$ is the irreducible character of $\sym N$ corresponding to $\lambda$
and $f^\lambda$ is the number of standard tableaux with shape $\lambda$.
We denote by $K_{\lambda\mu}$ the Kostka number,
that is, the number of tableaux with shape $\lambda$ and weight $\mu$. 
Note that $f^\lambda=K_{\lambda(1^N)}=\chi^\lambda(1)$ for each $\lambda\vdash N$.
By the hook formula for $f^\lambda$ and the definition of $f_\lambda(x)$, we have
\begin{equation}\label{eq:values at -1/k and 1/n}
f_{(k^n)}(-1/k)=\frac{(kn)!}{k^{kn}}\frac1{f^{(k^n)}},\qquad
f_{(k^n)}(1/n)=\frac{(kn)!}{n^{kn}}\frac1{f^{(k^n)}}.
\end{equation}

For each pair $\lambda,\mu$ of partitions of $N$, define
\begin{equation*}
\omega^\lambda_\mu(x):=\frac1{\mu!}\sum_{\tau\in\sym\mu}\chi^\lambda(x\tau)
\qquad(x\in\sym N),
\end{equation*}
where $\sym\mu=\sym{\mu_1}\times\sym{\mu_2}\times\dots$ is the Young subgroup associated to $\mu$
and $\mu!=\mu_1!\mu_2!\dots$ is its cardinality.
Here we regard the $i$-th component $\sym{\mu_i}$ in $\sym\mu$
as a subgroup of $\sym N$ consisting of permutations of the $\mu_i$ letters
$m+1,\dots,m+\mu_i$ with $m=\sum_{j<i}\mu_j$.
It is immediate to see that $\omega^\lambda_\mu$ is $\sym\mu$-biinvariant function on $\sym N$.
It is well known that
\begin{equation}\label{eq:Kostka}
K_{\lambda\mu}
=\frac1{\mu!}\sum_{\tau\in\sym\mu}\chi^\lambda(\tau)
=\omega^\lambda_\mu(1)
\end{equation}
for $\lambda,\mu\vdash N$.

Let $*$ be the convolution product defined by 
\begin{equation*}
(\phi_1*\phi_2)(x)=\sum_{\sigma\in\sym N}\phi_1(x\sigma)\phi_2(\sigma^{-1})
\end{equation*}
for $\phi_1,\phi_2\colon\sym N\to\C$.
Recall that the irreducible characters satisfy
\begin{equation}\label{eq:orthogonality}
\chi^\lambda*\chi^\rho=\delta_{\lambda\rho}\frac{N!}{f^\lambda}\chi^\lambda\qquad(\lambda,\rho\vdash N).
\end{equation}

We need the following Fourier expansion formula.
\begin{lem}[Fourier expansion of $\alpha^{\nu(\cdot)}$]\label{lem:Fourier}
\begin{equation}\label{eq:expansion_of_alpha^nu}
\alpha^{\nu(\sigma)}
=\frac1{N!}\sum_{\lambda\vdash N}f^\lambda f_\lambda(\alpha)\chi^\lambda(\sigma)
\qquad(\sigma\in\sym N).
\end{equation}
\end{lem}
By \eqref{eq:orthogonality} and \eqref{eq:expansion_of_alpha^nu}, we have
\begin{equation*}
\alpha^{\nu(\cdot)}*\beta^{\nu(\cdot)}
=\frac1{N!}\sum_{\lambda\vdash N}f^\lambda f_\lambda(\alpha)f_\lambda(\beta)\chi^\lambda.
\end{equation*}
Hence it follows that
\begin{equation*}
\adet[\alpha,\beta]X
=\frac1{N!}\sum_{\lambda\vdash N}f^\lambda f_\lambda(\alpha)f_\lambda(\beta)\Imm_\lambda X,
\end{equation*}
where $\Imm_\lambda X$ is the \emph{immanant} associated to $\lambda$ defined by
\begin{equation}\label{eq:definition of immanant}
\Imm_\lambda X=\sum_{\sigma\in\sym N}\chi^\lambda(\sigma) x_{\sigma(1)1}x_{\sigma(2)2}\dots x_{\sigma(N)N}
=\sum_{\sigma\in\sym N}\chi^\lambda(\sigma) \adet[0](P(\sigma^{-1})X).
\end{equation}
For a partition $\mu=(\mu_1,\mu_2,\dots,\mu_l)\vdash N$, define
\begin{equation*}
\1_\mu:=\begin{pmatrix}
\1_{\mu_1} \\ & \1_{\mu_2} \\ && \ddots \\ &&& \1_{\mu_l}
\end{pmatrix}.
\end{equation*}
For example, we have $\1_{(k^n)}=I_n\otimes\1_k$.
We have
\begin{equation*}
\omega^\lambda_\mu(g)=\frac1{\mu!}\Imm_\lambda\bigl(P(g)\1_\mu\bigr)
\end{equation*}
since
\begin{equation*}
\adet[0](P(g)\1_\mu)
=\begin{cases}
1 & g\in\sym\mu, \\
0 & g\notin\sym\mu.
\end{cases}
\end{equation*}
Thus it follows that
\begin{equation}\label{eq:adet_ab_1rho}
\adet[\alpha,\beta](P(g)\1_\mu)
=\frac{\mu!}{N!}\sum_{\lambda\vdash N}f^\lambda f_\lambda(\alpha)f_\lambda(\beta)\omega^\lambda_\mu(g)
\qquad(g\in\sym N).
\end{equation}
As a particular case, we have
\begin{equation}\label{eq:adet_ab_Ix1}
\adet[\alpha,\beta](I_n\otimes\1_k)
=\frac{(k!)^n}{(kn)!}\sum_{\lambda\vdash kn}f^\lambda K_{\lambda,(k^n)}f_\lambda(\alpha)f_\lambda(\beta).
\end{equation}
by putting $N=kn$, $\mu=(k^n)$ and $g=1$ because of \eqref{eq:Kostka}.

Now we further assume that $\beta=-1/k$ in \eqref{eq:adet_ab_Ix1}.
By definition, we have $f_\lambda(-1/k)=0$ unless $\lambda_1\le k$.
On the other hand, $K_{\lambda,(k^n)}=0$ unless $\len\lambda\le n$.
Therefore, the summand
in \eqref{eq:adet_ab_Ix1}
vanishes unless $\lambda=(k^n)$.
Thus it follows that
\begin{equation*}
F(\alpha;\I_{n,k})
=\adet[\alpha,-1/k](I_n\otimes\1_k)
=\frac{(k!)^n}{(kn)!}f^{(k^n)}f_{(k^n)}(\alpha)f_{(k^n)}(-1/k)
=\frac{(k!)^n}{k^{kn}}f_{(k^n)}(\alpha),
\end{equation*}
where we use \eqref{eq:values at -1/k and 1/n} in the last equality.
This completes the proof of the theorem.


\section{Proofs of the lemmas}\label{sec:proofs of lemmas}

Here we prove Lemmas used in the previous section.
The proof of Lemma \ref{lem:char_of_wrdet} below
is different from the one given in \cite{KW2008},
and is rather elementary.
The proof of Lemma \ref{lem:weak alternating} is just a revision of the one given in \cite{KW2008}.
Lemma \ref{lem:Fourier} is prove by using Okounkov-Vershik theory \cite{OV1996}
on representations of symmetric groups.

\subsection{Proof of Lemma \ref{lem:char_of_wrdet}}

Let $f\colon M_{kn,n}\to\C$ be a function satisfying the conditions (W1)--(W3).
We put
$$
\phi(\sigma)=f(P(\sigma)\I_{n,k})
$$
for $\sigma\in\sym{kn}$.
Notice that
$$
\phi(\tau\sigma\tau')=\phi(\sigma)
\qquad(\tau,\tau'\in\sym k^n,\;\sigma\in\sym{kn})
$$
by (W3) and the invariance $P(\tau')\I_{n,k}=\I_{n,k}$.
Let $I$ and $J$ be fixed complete systems of representatives of the coset $\sym{kn}/\sym k^n$
and the double coset $\sym k^n\backslash\sym{kn}/\sym k^n$ respectively,
and define $I(\sigma)$ for each $\sigma\in J$ to be the subset of $I$ such that
$\coprod_{\tau\in I(\sigma)}\tau\sym k^n=\sym k^n\sigma\sym k^n$.
Notice that $\phi(\tau)=\phi(\sigma)$ for each $\tau\in I(\sigma)$.

By (W1), we have
\begin{equation*}
f(A)=\sum_{1\le j_1,\dots,j_{kn}\le n}a_{1j_1}\dots a_{kn,j_{kn}}f\bigl(\tr(\ve_{j_1}~\dots~\ve_{j_{kn}})\bigr)
\end{equation*}
for $A=(a_{ij})\in M_{kn,n}$,
where $\ve_1,\ve_2,\dots,\ve_n$ are the standard basis vectors of $\C^{n}$.
By (W2), $f(\tr(\ve_{j_1}~\dots~\ve_{j_{kn}}))$ vanishes unless the matrix rank of
$(\ve_{j_1}~\dots~\ve_{j_{kn}})$ equals $n$,
or $(j_1,\dots,j_{kn})$ is a permutation of
$(\overbrace{1,\dots,1}^k,\dots,\overbrace{n,\dots,n}^k)$.
Hence it follows that
\begin{equation*}
f(A)=\sum_{\tau\in I}\phi(\tau)\;a_{\tau(1)1}\dots a_{\tau(kn)n}
=\sum_{\sigma\in J}\phi(\sigma)\sum_{\tau\in I(\sigma)}a_{\tau(1)1}\dots a_{\tau(kn)n}.
\end{equation*}
If we take $A=\I_{n,k}X$, $X=(x_{ij})\in M_n$, then we have
\begin{equation*}
f(\I_{n,k}X)
=\sum_{\sigma}\phi(\sigma)
\sum_{\tau\in I(\sigma)}x^{M(\tau)}
=\sum_{\sigma}\phi(\sigma)
\#I(\sigma) x^{M(\sigma)},
\end{equation*}
where
$$
x^{M(\sigma)}=\prod_{i,j}x_{ij}^{m_{ij}(\sigma)},\qquad
m_{ij}(\sigma)=\#\set{s}{(i-1)k<s\le ik,\;(j-1)k<\sigma(s)\le jk}.
$$
Notice that $M(\sigma)$ depends only on the double coset $\sym k^n\sigma\sym k^n$.
On the other hand, by (W2), we have $f(\I_{n,k}X)=f(\I_{n,k})(\det X)^k$.
Hence we get
\begin{equation}\label{eq:phi is equal to the coefficient}
\phi(\sigma)=\frac{f(\I_{n,k})}{\# I(\sigma)}
\times\text{coefficient of $x^{M(\sigma)}$ in }
(\det X)^k.
\end{equation}
As a result, we have
\begin{equation*}
f(A)=\frac{f(\I_{n,k})}{\wrdet_k \I_{n,k}}\wrdet_k A
\end{equation*}
as desired.

\subsection{Proof of Lemma \ref{lem:weak alternating}}

It is enough to prove that
the sum \eqref{eq:unsigned average of adet}
is divisible by $(1+\alpha)\dots(1+(k-1)\alpha)$.
We see that \eqref{eq:unsigned average of adet} is equal to
$$
\sum_{\tau\in\sym N}\left(\sum_{\sigma\in\sym k}\alpha^{\nu(\tau\sigma)}\right)\prod_{i=1}^N a_{\tau(i)i}.
$$
For each $\tau\in\sym N$, there uniquely exists $\tau_0\in\sym k$ such that
$\nu(\tau\sigma)=\nu(\tau\tau_0^{-1})+\nu(\tau_0\sigma)$ for any $\tau\in\sym k$.
In fact, if we define $g_i$ and $\tau_i$ for $i=n,n-1,\dots,1$ recursively by
$$
\tau_n=\tau;\qquad
g_i=(i~\tau_i(i)),\quad \tau_{i-1}=g_i\tau_i,
$$
then we have $\tau_0=g_kg_{k-1}\dots g_1$.
It follows that
$$
\sum_{\sigma\in\sym k}\alpha^{\nu(\tau\sigma)}
=\alpha^{\nu(\tau\tau_0^{-1})}\adet\1_k
=\alpha^{\nu(\tau\tau_0^{-1})}(1+\alpha)\dots(1+(k-1)\alpha).
$$
Hence the sum \eqref{eq:unsigned average of adet} is divisible by $(1+\alpha)\dots(1+(k-1)\alpha)$.

\subsection{Proof of Lemma \ref{lem:Fourier}}

Let $X_1,\dots,X_N$ be the Jucyc-Murphy elements of the group algebra $\C\sym N$:
\begin{equation*}
X_k=(1~k)+(2~k)+\dots+(k-1~k)\qquad(1\le k\le N).
\end{equation*}
It is elementary to see that
\begin{equation*}
\phi:=
\sum_{\sigma\in\sym N}\alpha^{\nu(\sigma)}\sigma
=(1+\alpha X_1)(1+\alpha X_2)\dots(1+\alpha X_N),
\end{equation*}
which is central since $\nu$ is a class function.
So it is a linear combination of the projections
\begin{equation*}
P_\lambda:=\frac1{N!}\sum_{\sigma\in\sym N}\chi^\lambda(\sigma)\sigma\qquad(\lambda\vdash N).
\end{equation*}

For each partition $\lambda\vdash N$,
let $\{v_T\}_{T\in\STab(\lambda)}$ be the Gelfand-Tsetlin basis (or Young basis)
of the irreducible representation $\Smod\lambda$ of $\sym N$
associated to $\lambda$, where $\STab(\lambda)$ is the set of standard tableaux with shape $\lambda$.
It is known that
if the number written in the $(i,j)$-position of $T$ is $k$, then
\begin{equation*}
X_kv_T=(j-i)v_T.
\end{equation*}
Hence it follows that
\begin{equation*}
\phi v_T=\prod_{(i,j)\in\lambda}(1+(j-i)\alpha)v_T=f_\lambda(\alpha)v_T
\end{equation*}
for any $T\in\STab(\lambda)$.
Thus we have
\begin{equation*}
\Tr\phi\big|_{\Smod\lambda}=f^\lambda f_\lambda(\alpha),
\end{equation*}
so that we get
\begin{equation*}
\phi
=\sum_{\lambda\vdash N}\Tr\phi\big|_{\Smod\lambda}P_\lambda
=\sum_{\sigma\in\sym N}\Bigl(\frac1{N!}\sum_{\lambda\vdash N}f^\lambda f_\lambda(\alpha)\chi^\lambda(\sigma)\Bigr)\sigma
\end{equation*}
as desired.

\section{Corollaries of the discussion}


We obtain the following ``determinantal'' formula of
the values of $\omega^\lambda_\mu$ and the Kostka numbers $K_{\lambda\mu}$ for rectangular-shaped Young diagrams $\lambda$
as a byproduct of the discussion above.

\begin{cor}\label{cor:values_of_omega}
For any $g\in\sym{kn}$ and $\mu\vdash kn$, it holds that
\begin{equation}\label{eq:omega_for_rect_shape}
\omega^{(k^n)}_\mu(g)
=\frac{f^{(k^n)}}{\mu!}\frac{\adet[-1/k,1/n](P(g)\1_\mu)}{\adet[-1/kn]\1_{kn}}.
\end{equation}
In particular, it holds that
\begin{equation}\label{eq:Kostka_for_rect_shape}
K_{(k^n)\mu}
=\frac{f^{(k^n)}}{\mu!}\frac{\adet[-1/k,1/n]\1_\mu}{\adet[-1/kn]\1_{kn}}.
\end{equation}
\end{cor}

\begin{proof}
We first notice that
$\adet[-1/kn]\1_{kn}=(kn)!/(kn)^{kn}$ by \eqref{eq:adet of all-one matrix}.
Putting $N=kn$, $\alpha=-1/k$ and $\beta=1/n$ in \eqref{eq:adet_ab_1rho}, we have
\begin{equation*}
\adet[-1/k,1/n](P(g)\1_\mu)
=\frac{\mu!}{(kn)!}\sum_{\lambda\vdash kn}f^\lambda f_\lambda(-1/k)f_\lambda(1/n)\omega^\lambda_\mu(g).
\end{equation*}
Since $f_\lambda(-1/k)=0$ unless $\lambda_1\le k$ and $f_\lambda(1/n)=0$ unless $\len\lambda\le n$,
the summand in the right-hand side of the equation above vanishes unless $\lambda=(k^n)$.
Therefore we have
\begin{multline*}
\adet[-1/k,1/n](P(g)\1_\mu)
=\frac{\mu!}{(kn)!}f^{(k^n)} f_{(k^n)}(-1/k)f_{(k^n)}(1/n)\omega^{(k^n)}_\mu(g) \\
=\frac{\mu!}{f^{(k^n)}}\frac{(kn)!}{(kn)^{kn}}\omega^{(k^n)}_\mu(g)
=\frac{\mu!}{f^{(k^n)}}\adet[-1/kn](\1_{kn})\omega^{(k^n)}_\mu(g),
\end{multline*}
which implies \eqref{eq:omega_for_rect_shape}.
The equation \eqref{eq:Kostka_for_rect_shape} is readily obtained by putting $g=1$ in \eqref{eq:omega_for_rect_shape}.
\end{proof}

By putting $\mu=(1^{kn})$ in Corollary \ref{cor:values_of_omega},
we have a formula of irreducible characters for rectangular diagrams.
\begin{cor}\label{cor:values_of_chi}
It holds that
\begin{equation}\label{eq:chi_for_rect_shape}
\frac{\chi^{(k^n)}(g)}{f^{(k^n)}}=\frac{\adet[-1/k,1/n]P(g)}{\adet[-1/kn]\1_{kn}}
\end{equation}
for $g\in\sym{kn}$.
\qed
\end{cor}

\begin{rem}
Assume that $m\le N=kn$.
Let $\iota\colon\sym m\to\sym N$ be the natural inclusion.
Then we have
\begin{equation*}
\frac{\adet[\alpha,\beta]P(\iota(w))}{\adet[\alpha,\beta]I_N}
=\frac{\adet[\alpha,\beta]P(w)}{\adet[\alpha,\beta]I_m}
=\sum_{\sigma\in\sym m}\alpha^{\nu(w\sigma)}\beta^{\nu(\sigma^{-1})}.
\end{equation*}
Hence, for $w\in\sym m$, the formula \eqref{cor:values_of_chi} gives Stanley's formula \cite{S2003}
\begin{equation}\label{eq:Stanley's formula}
\frac{N!}{(N-m)!}\frac{\chi^{(k^n)}(\iota(w))}{\chi^{(k^n)}(1)}
=(-1)^m \sum_{\sigma\in\sym m}(-k)^{\kappa(w\sigma)}n^{\kappa(\sigma^{-1})},
\end{equation}
where $\kappa(\sigma)$ denotes the number of disjoint cycles in $\sigma$.
\end{rem}

We look at another particular case where $\mu=(k^n)$.
For each $\lambda\vdash N$, we put $\omega^\lambda:=\omega^\lambda_\lambda$.

\begin{cor}\label{cor:values_of_zsf}
Let $n,k$ be positive integers. For any $g\in\sym{kn}$,
\begin{equation*}
\omega^{(k^n)}(g)=\frac{\wrdet_k\bigl(P(g)\I_{n,k}\bigr)}{\wrdet_k\I_{n,k}}
\end{equation*}
holds.
\end{cor}

\begin{proof}
As we see in the proof of Corollary \ref{cor:values_of_omega}, we have
\begin{equation*}
\adet[-1/k,1/n](P(g)\1_{(k^n)})
=\frac{(k!)^n}{(kn)!}f_{(k^n)}(1/n)\omega^{(k^n)}(g)
=\omega^{(k^n)}(g)f_{(k^n)}(1/n)\wrdet_k\I_{n,k}.
\end{equation*}
On the other hand, by Theorem \ref{thm:main_result}, we have
\begin{equation*}
\adet[-1/k,1/n](P(g)\1_{(k^n)})
=\adet[-1/k,1/n]\Bigl(\varpi_k\bigl(P(g)\I_{n,k}\bigr)\Bigr)
=f_{(k^n)}(1/n)\wrdet_k\bigl(P(g)\I_{n,k}\bigr).
\end{equation*}
Combining these two, we obtain the desired conclusion.
\end{proof}

\begin{rem}
By Corollary \ref{cor:values_of_zsf} and \eqref{eq:phi is equal to the coefficient}, we have
\begin{equation*}
\omega^{(k^n)}(\sigma)=
\frac{\text{coefficient of $x^{M(\sigma)}$ in }(\det X)^k}
{\gind{\sym k^n}{\sym k^n\cap\sigma^{-1}\sym k^n\sigma}}
\end{equation*}
for $\sigma\in\sym{kn}$.
\end{rem}

\begin{flushleft}
\bigskip

Kazufumi Kimoto \par
Department of Mathematical Sciences, \par
University of the Ryukyus \par
1 Senbaru, Nishihara, Okinawa 903-0213 JAPAN \par
\texttt{kimoto@math.u-ryukyu.ac.jp}
\end{flushleft}

\end{document}